\documentclass[10pt, reqno]{amsart}
\usepackage{amsmath, amsthm, amscd, amsfonts, amssymb, graphicx, color}
\usepackage[bookmarksnumbered, colorlinks, plainpages]{hyperref}

\makeatletter \oddsidemargin.9375in \evensidemargin \oddsidemargin
\newtheorem{theorem}{Theorem}[section]
\newtheorem{lemma}[theorem]{Lemma}

\newtheorem{corollary}[theorem]{Corollary}
\theoremstyle{definition}

\theoremstyle{remark}

\numberwithin{equation}{section}

\begin{document}
\setcounter{page}{1}


\title[ADDITIVITY OF MULTIPLICATIVE (GENERALIZED) MAPS OVER RINGS]{ADDITIVITY OF MULTIPLICATIVE (GENERALIZED) MAPS OVER RINGS}
\author[Sk Aziz,  Arindam Ghosh and Om Prakash]{Sk Aziz,  Arindam Ghosh and Om Prakash$^{*}$}
\thanks{{\scriptsize
\hskip -0.4 true cm MSC(2020): 16N60, 16W25, 16W55.
$*$Corresponding author}}

	\begin{abstract}
	In this paper, we mainly prove some results on the additivity of maps over rings under certain conditions. First, we discuss a special cases of MARTINDALE III's theorem of \cite{1969M} as a bijective map $\varphi$ over a ring $R$ with a non-trivial idempotent satisfying $\varphi(ab)=\varphi(a)\varphi(b)$ for all $a, b\in R$, is additive. Then we prove that a map $D$ on $R$ satisfying $D(ab)=D(a)b+\varphi(a) D(b)$ for all $a,b\in R$, where $\varphi$ is the map mentioned above, is additive. Finally, we establish that if a map $g$ over $R$ satisfies $g(ab)=g(a)b+\varphi(a)D(b),$ for all $a,b\in R$ and the maps $\varphi$ and $D$ are mentioned above, then $g$ is additive.	
	
\end{abstract}

\maketitle

\section{Introduction}
Recall that a ring $R$ is a prime ring if $aRb=0$ for some $a, b\in R$ implies $a=0$ or $b=0$. Similarly, a ring $R$ is called semi-prime if $aRa=0$ for some $a\in R$ implies that $a=0$. Let $n$ be a positive integer. A ring $R$ is called $n$-torsion free if $na=0$ for some $a\in R$ implies that $a=0$. Finally, let $R=(R,+,\cdot)$ be a ring. The opposite ring of $R$, denoted by $R^{\circ}=(R,+,\star)$, is the same as $R$ except that the multiplication $\star$ is defined by $a\star b=b \cdot a$, for all $a, b\in R$.

Throughout this paper, $R$ is an associative ring with the identity element $1$ and a non-trivial idempotent element $e$. Let $e_1=e$ and $e_2=1-e_1$. Then $R=R_{11} \oplus R_{12} \oplus R_{21} \oplus R_{22}$, where $R_{ij}=e_i R e_j$ and $1\leq i, j \leq 2$. A map $f: R \rightarrow R$ is said to be additive if it satisfies $f(a+b)=f(a)+f(b)$ for all $a,b \in R$. An additive bijective map $\varphi: R \rightarrow R$ is called an automorphism if $\varphi(ab)=\varphi(a)\varphi( b)$ for all $a,b \in R$, and an anti-automorphism if $\varphi(ab)=\varphi(b)\varphi( a)$ for all $a,b \in R$. Without assuming the additivity condition of $\varphi$, the map is known as a multiplicative automorphism and multiplicative anti-automorphism, respectively. An additive map $d: R \rightarrow R$ is said to be a derivation if it satisfies $d(ab)=d(a) b+a d(b)$ for all $a,b\in R$.

In 1957, Herstein proved in \cite{herstein1957jordan} that any Jordan derivation (a generalization of ordinary derivation) over a prime ring becomes an ordinary derivation with some torsion restriction of the ring. Let $\varphi$ be an automorphism of $R$. An additive map $D: R \rightarrow R$ is said to be a skew derivation if	$D(ab)=D(a) b+\varphi(a) D(b)$,
for all $a,b \in R$. It is also known as a $\varphi$-derivation. For example, we consider the identity homomorphism $1_R$ on $R$. Then the map $\varphi-1_R$ is an example of a skew derivation. Similarly, we can define a generalized skew derivation. An additive map $g: R \rightarrow R$ is said to be a generalized skew derivation if $g(ab)=g(a)b+\varphi(a) D(b),$
for all $a, b \in R$, where $\varphi$ is an automorphism of $R$ and $D$ is a skew derivation associated with $g$. Thus, we have the above two derivations as a generalization of derivation and automorphism.

Researchers have studied skew derivations in ring theory and the theory of operator algebras. If we remove the condition of additivity, then it is called a multiplicative derivation. Similarly, $D$ and $g$ are called multiplicative skew derivations and multiplicative generalized skew derivations, respectively, without assuming the additivity of the maps. It is natural to ask when are some multiplicative maps additive.

In 1948  \cite{rickart1948one}, Rickart first raised this question. He showed that a bijective and multiplicative onto map $h: B \rightarrow R^{\prime}$ is additive, where $B$ is a Boolean ring and $R^{\prime}$ is any ring. He also proved that any bijective multiplicative mapping from a ring $R$ onto a ring $S$ is additive, where $R$ contains a family of minimal ideals satisfying certain conditions.

In 1958 \cite{johnson1958rings}, Johnson extended Rickart's result to a larger class of rings. In 1969  \cite{martindale1969multiplicative}, Martindale proved that every multiplicative isomorphism from a ring $R$ onto a ring $R^{\prime}$ is additive. In 1991, using Martindale's conditions,  Daif \cite{daif1991multiplicative} proved that any multiplicative derivation on a ring $R$ is additive.

In 2009 \cite{wang2009additivity}, Wang proved that any multiplicative isomorphism from a ring $R$ onto a ring $S$ is additive, where $R$ contains a family of idempotents satisfying certain conditions. In 2012 \cite{jing2012additivity}, Jing and Lu proved that every multiplicative Jordan derivation and Jordan triple derivation over a ring $R$ with a non-trivial idempotent satisfying certain conditions is additive.

In 2014 \cite{ferreira2014multiplicative}, Ferreira proved that every $m$-multiplicative isomorphism from a triangular $n$-matrix ring onto another ring is additive, and every multiplicative $m$-derivation over any triangular $n$-matrix ring is additive, where the triangular $n$-matrix ring satisfies certain conditions in both cases. In 2015 \cite{ferreira2015jordan}, Ferreira also proved that every multiplicative Jordan derivation on a triangular ring with certain conditions is additive.

In 2017 \cite{yadav2017additivity}, Yadav and Sharma proved that any multiplicative generalized Jordan derivation on a ring $R$ with a non-trivial idempotent is additive. For more results, see \cite{a23,a24}. This result raises the question: "When are multiplicative automorphisms, multiplicative skew derivations, and multiplicative generalized skew derivations additive?" In this paper, we find an affirmative answer to this question. Also, in the case of a multiplicative skew derivation $D$, we consider $\varphi$ as a multiplicative automorphism. Similarly, for a multiplicative generalized skew derivation $g$, we think of $D$ as a multiplicative skew derivation.
\\

Let $R$ be a ring satisfying the following condition.
\begin{equation}
	\tag{A}
	\text{If} ~a_{i j} x_{j k}=0~\text{ for all} ~x_{j k} \in R_{j k}~ \text{and for some }~p_{ij}\in R_{ij}, ~\text{then}~ a_{i j}=0.
	\label{eqA}
\end{equation}
In the following section, we assume that $R$ satisfies condition \eqref{eqA} without explicitly mentioning it.

In this paper, we investigate the relationship between the multiplicative and additive structures of a ring $R$ with respect to certain classes of endomorphisms. Specifically, we consider the automorphisms, skew derivations, and generalized skew derivations on $R$ and study their interplay with the additive structure of $R$.

In Section 2, we prove that the multiplicative automorphisms on $R$ are necessarily additive. It provides a deep insight into the algebraic structure of $R$ and sheds light on the relationship between its additive and multiplicative operations. Section 3 turns our attention to skew derivations on $R$. Skew derivations are linear maps on $R$ that satisfy a particular multiplicative property. We prove that every multiplicative skew derivation on $R$ is necessarily additive and provide a new perspective on the structure of skew derivations and their relationship with the additive structure of $R$.
Section 4 extends our analysis to the more general case of multiplicative generalized skew derivations on $R$. These are maps that satisfy a weaker multiplicative condition than skew derivations. We prove that every multiplicative generalized skew derivation on $R$ is necessarily additive. Our results in this section deepen our understanding of the relationship between the additive and multiplicative structures of $R$ in the context of generalized skew derivations.

Overall, our results in this paper shed light on the interplay between the additive and multiplicative structures of a ring and provide new insights into the algebraic structure of the rings and their endomorphisms. Our analysis of automorphisms, skew derivations, and generalized skew derivations on $R$ highlights the crucial role played by the additive structure in understanding the multiplicative structure of a ring.

\section{Multiplicative Automorphism}

\begin{theorem}
	\label{thm2.1}
	If $\varphi$ is a multiplicative automorphism on $R$, then $\varphi$ is additive. Moreover, $\varphi$ is an automorphism on $R$.
\end{theorem}

\begin{proof}
	Since, $R$ is an associative ring with the identity element $1$ and a non-trivial idempotent element $e$, and satisfies the following
	condition \eqref{eqA}.
	Then, $R$ satisfies conditions $(1), (2)$ and $(3)$ of MARTINDALE III's Theorem \cite{1969M}. Hence, any
	multiplicative isomorphism $\varphi$ of $R$ onto an arbitrary ring $S$ is additive. Thus, this is a special care of that result.
\end{proof}

With the motivation of the Corollaries given in \cite{1969M}, we have the following results (corollaries).
\begin{corollary}
	\label{cor2.9}
	If $\varphi$ is a multiplicative automorphism on a prime ring $R$ with a non-trivial idempotent $e$, then $\varphi$ is additive. Moreover, $\varphi$ is an automorphism on $R$.
\end{corollary}

\begin{proof}
	Let $e_1=e$ and $e_2=1-e_1$. Then $R_{ij}=e_i R e_j$. Then
	\begin{align*}
		& p_{ij}x_{jk}=0,~\text{for some}~p_{ij}\in R_{ij}~\text{and for all}~x_{jk}\in R_{jk}\\
		\implies & e_i p e_j e_j x e_k=0, ~\text{for some}~p\in R~\text{and for all}~x\in R\\
		\implies & (e_i p e_j) x (e_k)=0, ~\text{and for all}~x\in R\\
		\implies & e_i p e_j=0 ~\text{(Since R is a prime ring)}\\
		\implies & p_{ij}=0.
	\end{align*}
	Thus, $R$ satisfies condition \eqref{eqA}. Hence, we get the desired result by Theorem \ref{thm2.1}.
\end{proof}

\begin{corollary}
	\label{cor2.10}
	If $\varphi$ is a multiplicative anti-automorphism on a prime ring $R$ with a non-trivial idempotent $e$, then $\varphi$ is additive. Moreover, $\varphi$ is an anti-automorphism on $R$.
\end{corollary}

\begin{proof}
	Let $\tau: R \rightarrow R^{\circ}$ be the map defined by
	\begin{align*}
		\tau(a)=a,
	\end{align*}
	for all $a\in R$. Then $\tau$ is an anti-isomorphism. Let $\sigma=\tau \circ \varphi$. Then $\sigma: R \rightarrow R^{\circ}$ is a multiplicative isomorphism. Then $\sigma$ is additive by a result in \cite{1969M}. Therefore, $\varphi$ is additive (Since $\tau$ is additive and one-one).
\end{proof}

\section{Multiplicative Skew Derivation}
\begin{theorem}
	\label{thm3.1}
	If $D$ is a multiplicative skew derivation on $R$, then $D$ is additive. Moreover, $D$ is a skew derivation on $R$.
\end{theorem} 	

Let $\varphi$ be the associated multiplicative automorphism on $R$. Hence, by Theorem \ref{thm2.1}, $\varphi$ is additive. Before proving Theorem \ref{thm3.1}, we have several lemmas.

\begin{lemma}
	\label{lem3.2}
	\begin{align*}
		D(0)=0.
	\end{align*}
\end{lemma}

\begin{proof}
	\begin{align*}
		D(0)=D(0\cdot 0)=D(0)0+\varphi(0)D(0)=0 ~(\text{By Theorem \ref{thm2.1}}).
	\end{align*}
\end{proof}

\begin{lemma}
	\label{lem3.3}
	Let $p_{ij}, q_{ij} \in R_{ij}$. Then
	$$
	\begin{aligned}
		(i) &~ D (p_{11}+q_{12} )=D (p_{11} )+D (q_{12} ),\\
		(ii) &~  D (p_{22}+q_{21} )=D (p_{22} )+D (q_{21} ),\\
		(iii) &~  D (p_{11}+q_{21} )=D (p_{11} )+D (q_{21} ), \\
		(iv) &~  D (p_{22}+q_{12} )=D (p_{22} )+D (q_{12} ) .
	\end{aligned}
	$$
\end{lemma}

\begin{proof}
	Let $x_{22} \in R_{22}$. Then
	
	\begin{equation}
		\label{eq3.1}
		D ( (p_{11}+q_{12} ) x_{22} ) = D (p_{11}+q_{12} ) x_{22}+\varphi (p_{11}+q_{12} ) D (x_{22} ).
	\end{equation}
	On the other hand,	
	\begin{equation}
		\label{eq3.2}
		\begin{aligned}
			& D ( (p_{11}+q_{12} ) x_{22} )\\
			=&~D (q_{12} x_{22} ) \\
			=&~D (p_{11} x_{22} )+D (q_{12} x_{22} ) \\
			=&~D (p_{11} ) x_{22}+\varphi (p_{11} ) D (x_{22} )+D (q_{12} ) x_{22}+\varphi (q_{12} ) D (x_{22} ).
		\end{aligned}
	\end{equation}
	Comparing \eqref{eq3.1} and \eqref{eq3.2}, we get
	$$
	\begin{aligned}
		&{ [D (p_{11}+q_{12} )-D (p_{11} )-D (q_{12} ) ] x_{22}=0} ~\text{(Since} ~\varphi~\text{is additive)}\\
		\implies~& [D (p_{11}+q_{12} )-D (p_{11} )-D (q_{12} ) ]_{12} x_{22}=0\\
		\& ~&[D (p_{11}+q_{12} )-D (p_{11} )-D (q_{12} ) ]_{22} x_{22}=0.
	\end{aligned}
	$$
	By the assumption \eqref{eqA} on $R$,
	$$
	\begin{gathered}
		{ [D (p_{11}+q_{12} )-D (p_{11} )-D (q_{12} ) ]_{12}=0 } \\
		\& ~[D (p_{11}+q_{12} )-D (p_{11} )-D (q_{12} ) ]_{22}=0.
	\end{gathered}
	$$
	Let $x_{12} \in R_{12}$. Then
	\begin{equation}
		\label{eq3.3}
		D ( (p_{11}+q_{12} ) x_{12} ) = D (p_{11}+q_{12} ) x_{12}+\varphi (p_{11}+q_{12} ) D (x_{12}).
	\end{equation}
	Also,
	\begin{equation}
		\label{eq3.4}
		\begin{aligned}
			D ( (p_{11}+q_{12} ) x_{12} )
			&=D (p_{11} x_{12} ) \\
			&=D (p_{11} x_{12} )+D (q_{12} x_{12} ) \\
			&=D (p_{11} ) x_{12}+\varphi (p_{11} ) D (x_{12} )+D (q_{12} ) x_{12}+\varphi (q_{12} ) D (x_{12} ).
		\end{aligned}
	\end{equation}
	Comparing \eqref{eq3.3} and \eqref{eq3.4},
	$$
	[D (p_{11}+q_{12} )-D (p_{11} )-D (q_{12} ) ] x_{12}=0.
	$$
	Again, by the condition \eqref{eqA} on $R$, we have
	$$
	\begin{aligned}
		&{ [D (p_{11}+q_{12} )-D (p_{11} )-D (q_{12} ) ]_{11}=0} \\
		\&~ & [D (p_{11}+q_{12} )-D (p_{11} )-D (q_{12} ) ]_{21}=0.
	\end{aligned}
	$$
	Hence,
	\begin{align*}
		D (p_{11}+q_{12} )=D (p_{11} )+D (q_{12} ).
	\end{align*}
	Similarly, we can prove
	\begin{align*}
		D (p_{22}+q_{21} )=D (p_{22} )+D (q_{21} ).
	\end{align*}
	Now,
	\begin{equation}
		\label{eq3.5}
		D(e_1(p_{11}+q_{21}))=D(e_1)(p_{11}+q_{21})+\varphi(e_1)D(p_{11}+q_{21}).
	\end{equation}
	Also,
	\begin{equation}
		\label{eq3.6}
		\begin{aligned}
			D(e_1(p_{11}+q_{21}))&=D(e_1 p_{11})\\
			&=D(e_1 p_{11})+D(e_1 q_{21})\\
			&=D(e_1)p_{11}+\varphi(e_1)D(p_{11})+D(e_1)q_{21}+\varphi(e_1)D(q_{21}).
		\end{aligned}
	\end{equation}
	Comparing \eqref{eq3.5} and \eqref{eq3.6},		
	\begin{equation}
		\label{eq3.7}
		\varphi(e_1)(D(p_{11}+q_{21})-D(p_{11})-D(q_{21}))=0.
	\end{equation}
	Similarly,
	\begin{equation}
		\label{eq3.8}
		\varphi(e_2)(D(p_{11}+q_{21})-D(p_{11})-D(q_{21}))=0.
	\end{equation}
	Adding \eqref{eq3.7} and \eqref{eq3.8}, we have
	\begin{align*}
		&(\varphi(e_1)+\varphi(e_2))(D(p_{11}+q_{21})-D(p_{11})-D(q_{21}))=0\\
		&\implies D(p_{11}+q_{21})=D(p_{11})+D(q_{21}) ~(\text{By Lemma \ref{lem2.3}}).
	\end{align*}
	Similarly, we can prove
	\begin{align*}
		D (p_{22}+q_{12} )=D (p_{22} )+D (q_{12} ).
	\end{align*}
\end{proof}

\begin{lemma}
	\label{lem3.4}
	Let $p_{ij}$, $q_{ij}, c_{ij} \in R_{ij}$. Then
	$$
	\begin{aligned}
		(i) &~ D (p_{12}+q_{12} c_{22} )=D (p_{12} )+D (q_{12} c_{22} ), \\
		(ii) &~ D (p_{21}+q_{22} c_{21} )=D (p_{21} )+D (q_{22} c_{21} ).
	\end{aligned}
	$$
\end{lemma}

\begin{proof}
	Note that
	\begin{align*}
		p_{12}+q_{12} c_{22}= (e_1+q_{12} )(p_{12}+c_{22} ).
	\end{align*}		
	Therefore,	
	$$
	\begin{aligned}
		D (p_{12}+q_{12} c_{22} )
		=~& D ( (e_1+q_{12} ) (p_{12}+c_{22} ) ) \\
		=~& D (e_1+q_{12} ) (p_{12}+c_{22} )+\varphi (e_1+q_{12} ) D (p_{12}+c_{22} ) \\
		=~& { [D (e_1 )+D (q_{12} ) ] (p_{12}+c_{22} )+ [\varphi (e_1 )+\varphi (q_{12} ] } [D (p_{12} )+D (c_{22} ) ]\\
		&~\text{(By Lemma \ref{lem3.3})} \\
		=~& D (e_1 p_{12} )+D (e_1 c_{22} )+D (q_{12} c_{22} ) +D (q_{12} p_{12} ) \\
		=~& D (p_{12} )+D (q_{12} c_{22} )~\text{(By Lemma \ref{lem3.2})} .
	\end{aligned}
	$$
	Similarly, we can prove that
	\begin{align*}
		D (p_{21}+q_{22} c_{21} )=D (p_{21} )+D (q_{22} c_{21} ).
	\end{align*}
\end{proof}

\begin{lemma}
	\label{lem3.5}
	Let $p_{12}, q_{12} \in R_{12}$ and $p_{21}, q_{21} \in R_{21}$. Then
	$$
	\begin{aligned}
		(i) &~ D (p_{12}+q_{12} )=D (p_{12} )+D (q_{12} ), \\
		(ii) &~ D (p_{21}+q_{21} )=D (p_{21} )+D (q_{21} ).
	\end{aligned}
	$$
\end{lemma}

\begin{proof}
	Let $x_{22} \in R_{22}$. Then
	$$
	\begin{aligned}
		&~D(p_{12}+q_{12} ) x_{22}+\varphi (p_{12}+q_{12} ) D (x_{22} )\\
		=&~D [ (p_{12}+q_{12} ) x_{22} ]\\
		=&~D (p_{12} x_{22} + q_{12} x_{22} ) \\
		=&~D (p_{12} x_{22} )+D (q_{12} x_{22} ) ~\text{(By Lemma \ref{lem3.4})}\\
		=&~D (p_{12} ) x_{12}+\varphi (p_{12} ) D (x_{22} )+D (q_{12} ) x_{22}+\varphi (q_{12} ) D (x_{22} ).
	\end{aligned}
	$$
	Consequently, we obtain
	\begin{equation}
		\label{eq3.9}
		\begin{aligned}
			&[D (p_{12}+q_{12} )-D (p_{12} )-D (q_{12} ) ] x_{22}=0\\
			\implies ~&[D (p_{12}+q_{12} )-D (p_{12} )-D (q_{12} ) ]_{12}=0\\
			\&~ & [D (p_{12}+q_{12} )-D (p_{12} )-D (q_{12} ) ]_{22}=0.
		\end{aligned}
	\end{equation}
	Let $x_{12} \in R_{12}$. Then
	$$
	\begin{aligned}
		&D(p_{12}+q_{12} ) x_{12}+\varphi (p_{12}+q_{12} ) D (x_{12} )\\
		&=D [ (p_{12}+q_{12} ) x_{12} ] \\
		&=D(0)=0 \\
		&=D (p_{12} x_{12} )+D (q_{12} x_{12} ) \\
		&=D (p_{12} ) x_{12}+\varphi (p_{12} ) D (x_{12} )+D (q_{12} ) x_{12}+\varphi (q_{12} ) D (x_{12} ),
	\end{aligned}
	$$
	which gives
	\begin{equation}
		\label{eq3.10}
		\begin{aligned}
			&[D (p_{12}+q_{12} )-D (p_{12} )-D (q_{12} ) ] x_{12}=0\\
			\implies ~&[D (p_{12}+q_{12} )-D (p_{12} )-D (q_{12} ) ]_{11}=0\\
			\&~ & [D (p_{12}+q_{12} )-D (p_{12} )-D (q_{12} ) ]_{21}=0.
		\end{aligned}
	\end{equation}
	Hence, by \eqref{eq3.9} and \eqref{eq3.10},
	\begin{align*}
		D (p_{12}+q_{12} )=D (p_{12} )+D (q_{12} ).
	\end{align*}
	Similarly, we can get
	\begin{align*}
		D (p_{21}+q_{21} )=D (p_{21} )+D (q_{21} ).
	\end{align*}		
\end{proof}

\begin{lemma}
	\label{lem3.6}
	Let $p_{ii}, q_{ii} \in R_{ii}$, where $i \in \{1,2\}$. Then
	$$
	\begin{aligned}
		& (i) ~D (p_{11}+q_{11} )=D (p_{11} )+D (q_{11} ),\\
		& (ii) ~D (p_{22}+q_{22} )=D (p_{22} )+D (q_{22} ).
	\end{aligned}
	$$
\end{lemma}

\begin{proof} Let $x_{22} \in R_{22}$. Then
	$$
	\begin{aligned}
		&D (p_{11}+q_{11} ) x_{22}+\varphi (p_{11}+q_{11} ) D (x_{22} )\\
		=& D [ (p_{11}+q_{11} ) x_{22} ] \\
		=& D(0)=0 \\
		=& D (p_{11} x_{22} )+D (q_{11} x_{22} ) \\
		=& D (p_{11} ) x_{22}+\varphi (p_{11} ) D (x_{22} )+D (q_{11} ) x_{22}+\varphi (q_{11} ) D (x_{22} )
	\end{aligned}
	$$
	Therefore,
	\begin{align*}
		[D (p_{11}+q_{11} )-D (p_{11} )-D (q_{11} ) ] x_{22}=0,
	\end{align*}
	which implies
	\begin{equation}
		\label{eq3.11}
		\begin{gathered}
			{ [D (p_{11}+q_{11} )-D (p_{11} )-D (q_{11} ) ]_{12}=0} \\
			\& ~[D (p_{11}+q_{11} )-D (p_{11} )-D (q_{11} ) ]_{22}=0
		\end{gathered}
	\end{equation}
	Let $x_{12} \in R_{12}$. Then
	$$
	\begin{aligned}
		&D (p_{11}+q_{11}) x_{12}+\varphi (p_{11}+q_{11} ) D (x_{12} )\\
		& =D [( p_{11}+q_{11} ) x_{12} ]\\
		&= D (p_{11} x_{12}+q_{11} x_{12} ) \\
		&=D (p_{11} x_{12} )+D (q_{11} x_{12} )~\text{(By Lemma \ref{lem3.5})} \\
		&=D (p_{11} ) x_{12}+\varphi (p_{11} ) D (x_{12} )+D (q_{11} ) x_{12}+\varphi (q_{11} ) D (x_{12} )
	\end{aligned}
	$$
	Therefore,
	\begin{align*}
		[D (p_{11}+q_{11} )-D (p_{11} )-D (q_{11} ) ] x_{12}=0,
	\end{align*}
	which implies
	\begin{equation}
		\label{eq3.12}
		\begin{gathered}
			[D (p_{11}+q_{11} )-D (p_{11} )-D (q_{11} ) ]_{11}=0\\
			\&~ [D (p_{11}+q_{11} )-D (p_{11} )-D (q_{11} ) ]_{21}=0.
		\end{gathered}
	\end{equation}
	Hence, by \eqref{eq3.11} and \eqref{eq3.12},
	\begin{align*}
		D (p_{11}+q_{11} )=D (p_{11} )+D (q_{11} ).
	\end{align*}		
	Similarly, we can prove that
	\begin{align*}
		D (p_{22}+q_{22} )=D (p_{22} )+D (q_{22} ).
	\end{align*}		
\end{proof}

\begin{lemma}
	\label{lem3.7}
	Let $p_{11} \in R_{11},~ q_{12} \in R_{12},~ c_{21} \in R_{21},~ d_{22} \in R_{22}$. Then
	\begin{align*}
		D (p_{11}+q_{12}+c_{21}+d_{22} )= D (p_{11} )+D (q_{12} ) +D (c_{21} )+D (d_{22} ).
	\end{align*}		 		
\end{lemma}

\begin{proof}
	Let $x_{11} \in R_{11}$. Then
	$$
	\begin{aligned}
		& D (p_{11}+q_{12}+c_{21}+d _{22} ) x_{11} +\varphi (p_{11}+q_{12}+c_{21}+d_{22} ) D (x_{11} ) \\
		=&D [ (p_{11}+q_{12}+c_{21}+ d_{22} ) x_{11} ]\\
		=& D (p_{11} x_{11}+c_{21} x_{11} ) \\
		=& D (p_{11} x_{11} )+D (q_{12} x_{11} )+D (c_{21} x_{11} )+D (d_{22} x_{11} )\\
		&~\text{(By Lemma \ref{lem3.2} and \ref{lem3.3})} \\
		=& D (p_{11} ) x_{11}+\varphi (p_{11} ) D (x_{11} )+D (q_{12} ) x_{11}+\varphi (q_{12} ) D (x_{11} ) \\
		&+D (c_{21} ) x_{11}+\varphi (c_{21} ) D (x_{11} )+D (d _{22} ) x_{11}+\varphi (d_{22} ) D (x_{11} ).
	\end{aligned}
	$$
	Comparing both sides and using the additivity of $\varphi$, we have
	\begin{equation}
		\label{eq3.13}
		\begin{aligned}
			&  [D (p_{11}+q_{12}+c_{21}+d_{22} )-D (p_{11} )-D (q_{12} )-D (c_{21} )-D (d_{22} ) ] x_{11}=0\\
			\implies &   [D (p_{11}+q_{12}+c_{21}+d_{22} )-D (p_{11} )-D (q_{12} )-D (c_{21} )-D (d_{22} ) ]_{11}=0\\
			&  \& ~[D (p_{11}+q_{12}+c_{21}+d_{22} )-D (p_{11} )-D (q_{12} )-D (c_{21} )-D (d_{22} ) ] _{21}=0.
		\end{aligned}
	\end{equation}
	Similarly, by taking $x_{22}$ from $R_{22}$, we can get
	\begin{equation}
		\label{eq3.14}
		\begin{aligned}
			& [D (p_{11}+q_{12}+c_{21}+d_{22} )-D (p_{11} )-D (q_{12} )-D (c_{21} )-D (d_{22} ) ]_{12}=0\\
			& \& ~  [D (p_{11}+q_{12}+c_{21}+d_{22} )-D (p_{11} )-D (q_{12} )-D (c_{21} )-D (d_{22} ) ] _{22}=0.
		\end{aligned}
	\end{equation}
	Hence, by \eqref{eq3.13} and \eqref{eq3.14},
	\begin{align*}
		D (p_{11}+q_{12}+c_{21}+d_{22} )= D (p_{11} )+D (q_{12} ) +D (c_{21} )+D (d_{22} ).
	\end{align*}	
\end{proof}

\begin{proof}[Proof of Theorem \ref{thm3.1}]
	Let $a, b \in R$. Then
	\begin{align*}
		& a=p_{11}+p_{12}+p_{21}+p_{22},\\
		& b=q_{11}+q_{12}+q_{21}+q_{22},
	\end{align*}		
	for some $p_{ij}, q_{ij} \in R_{ij}$.
	Now,
	$$
	\begin{aligned}
		D(a+b) =&D (p_{11}+p_{12}+p_{21}+p_{22}+q_{11}+q_{12}+q_{21}+q_{22} ) \\
		=& D [ (p_{11}+q_{11} )+ (p_{12}+q_{12} )+ (p_{21}+q_{21} )+ (p_{22}+q_{22} ) ] \\
		=& D (p_{11}+q_{11} )+D (p_{12}+q_{12} )+D (p_{21}+q_{21} )+D (p_{22}+q_{22} ) \\
		&~\text{(By Lemma \ref{lem3.7})}\\
		=& D (p_{11} )+D (q_{11} )+D (p_{12} )+D (q_{12} )\\
		&+D (p_{21} ) +D (q_{21} )+D (p_{22} )+D (q_{22} ) ~\text{(By Lemma \ref{lem3.5} and \ref{lem3.6})}\\
		=& D (p_{11}+p_{12}+p_{21}+p_{22} ) +D(q_{11} +q_{12}+q_{21}+q_{22} ) \\
		&~\text{(By Lemma \ref{lem3.7})}\\
		=& D(a) + D(b).
	\end{aligned}
	$$
	Hence, $D$ is additive.
\end{proof}

\begin{corollary}
	\label{cor3.8}
	Let $R$ be a $2$-torsion free semi-prime ring with a non-trivial idempotent $e$ and satisfies the condition
	\begin{equation}
		\tag{B}
		a_{i i} x_{i j}=0,~ \text{for some }~a_{ii}\in R_{ii}~\text{and for all} ~x_{i j} \in R_{i j}~(i\neq j), ~\text{then}~ a_{i i}=0.
		\label{eqB}
	\end{equation}	
	Then every multiplicative skew derivation $D$ over $R$ is additive. Moreover, $D$ is a skew derivation on $R$.
\end{corollary}

\begin{proof}
	Since $R$ satisfies \eqref{eqB}, by Lemma $1.5$ in \cite{jing2012additivity}, $R$ also satisfies the condition \eqref{eqA}. Hence, we have the desired result by Theorem \ref{thm3.1}.
\end{proof}

\section{Multiplicative Generalized Skew Derivation}
\begin{theorem}
	\label{thm4.1}
	If $g$ is a multiplicative generalized skew derivation on $R$, then $g$ is additive. Moreover, $g$ is a generalized skew derivation on $R$.
\end{theorem} 	

Let $\varphi$ and $D$ be the associated multiplicative automorphism and associated multiplicative skew derivation on $R$, respectively. Hence, by Theorem \ref{thm2.1} and \ref{thm3.1}, $\varphi$ and $D$ are additive. Before proving Theorem \ref{thm4.1}, we have several lemmas.

\begin{lemma}
	\label{lem4.2}
	\begin{align*}
		g(0)=0.
	\end{align*}
\end{lemma}

\begin{proof}
	\begin{align*}
		g(0)=g(0\cdot 0)=g(0)0+\varphi(0)D(0)=0~(\text{By Lemma \ref{lem2.2} and \ref{lem3.2}}).
	\end{align*}
\end{proof}	

\begin{lemma}
	\label{lem4.3}
	\begin{align*}
		D(1)=0.
	\end{align*}
\end{lemma}

\begin{proof}
	\begin{align*}
		&g(1)=g(1\cdot 1)=g(1)1+\varphi(1)D(1)=g(1)+D(1)\\
		& \implies D(1)=0.
	\end{align*}
\end{proof}

\begin{lemma}
	\label{lem4.4}
	\begin{align*}
		g(e_1+e_2)=g(e_1)+g(e_2).
	\end{align*}
\end{lemma}

\begin{proof}
	\begin{align*}
		& g(e_1)=g(1\cdot e_1)=g(1)e_1+\varphi(1)D(e_1)\\
		~\&~&g(e_2)=g(1\cdot e_2)=g(1)e_2+\varphi(1)D(e_2).
	\end{align*}
	Adding these,
	\begin{align*}
		g(e_1)+g(e_2)&=g(1)(e_1+e_2)+\varphi(1)(D(e_1)+D(e_2))\\
		&=g(1)1+\varphi(1)D(e_1+e_2)\\
		&=g(1)+D(1)=g(1)~(\text{By Lemma \ref{lem4.3}})\\
		&=g(e_1+e_2).
	\end{align*}
\end{proof}

\begin{lemma}
	\label{lem4.5}
	Let $p_{ij} \in R_{ij}$ and $q_{ij} \in R_{ij}$. Then
	$$
	\begin{aligned}
		(i) & ~ g(p_{11}+q_{12} )=g (p_{11} )+g (q_{12} ),\\
		(ii)&~  g (p_{22}+q_{21} )=g (p_{22} )+g (q_{21} ),\\
		(iii) & ~  g (p_{11}+q_{21} )=g (p_{11} )+g (q_{21} ), \\
		(iv)&~  g (p_{22}+q_{12} )=g (p_{22} )+g (q_{12} ) .
	\end{aligned}
	$$
\end{lemma}

\begin{proof}
	Let $x_{22} \in R_{22}$. Then
	$$
	\begin{aligned}
		&g (p_{11}+q_{12} ) x_{22}+\varphi (p_{11}+q_{12} ) D (x_{22} ) \\
		&=g [ (p_{11}+q_{12} ) x_{22} ]\\
		&=g (q_{12} x_{22} ) \\
		&=g (p_{11} x_{22} )+g (q_{12} x_{22} ) ~(\text{By Lemma \ref{lem4.2}})\\
		&=g (p_{11} ) x_{22}+\varphi (p_{11} ) D (x_{22} )+g (q_{12} ) x_{22}+\varphi (q_{12} ) D (x_{22} ).
	\end{aligned}
	$$
	Comparing both sides,
	$$
	[g (p_{11}+q_{12} )-g (p_{11} )-g (q_{12} ) ] x_{22}=0~(\text{By Theorem \ref{thm2.1}}).
	$$
	Using the assumption \eqref{eqA} on $R$, we get
	\begin{equation}
		\label{eq4.1}
		\begin{aligned}
			& { [g (p_{11}+q_{12} )-g (p_{11} )-g (q_{12} ) ]_{12}=0 } \\
			\& & ~{ [g (p_{11}+q_{12} )-g (p_{11} )-g (q_{12} ) ]_{22}=0 . }
		\end{aligned}
	\end{equation}
	Similarly, for $x_{12} \in R_{12}$, we have
	$$
	\begin{aligned}
		&g (p_{11}+q_{12} ) x_{12}+\varphi (p_{11}+q_{12} ) D (x_{12} )\\
		&=g [ (p_{11}+q_{12} ) x_{12} ]\\
		& =g (p_{11} x_{12} )\\
		&=g (p_{11} x_{12} )+g (q_{12} x_{12} )\\
		&=g (p_{11} ) x_{12}+\varphi (p_{11} ) D (x_{12} )+g (q_{12} ) x_{12}+\varphi (q_{12} ) D (x_{12} ).
	\end{aligned}
	$$
	This yields that,
	\begin{equation}
		\label{eq4.2}
		\begin{aligned}
			& [g (p_{11}+q_{12} )-g (p_{11} )-g (q_{12} ) ] x_{12}=0 \\
			& \implies [g (p_{11}+q_{12} )-g (p_{11} )-g (q_{12} ) ]_{11}=0 \\
			&\&~ [g (p_{11}+q_{12} )-g (p_{11} )-g (q_{12} ) ]_{21}=0.
		\end{aligned}
	\end{equation}
	Hence, by \eqref{eq4.1} and \eqref{eq4.2},
	\begin{align*}
		g (p_{11}+q_{12} )=g (p_{11} )+g (q_{12} ).
	\end{align*}
	Similarly, we can prove that
	\begin{align*}
		g (p_{22}+q_{21} )=g (p_{22} )+g (q_{21} ).
	\end{align*}
	Now,
	\begin{align*}
		&g(p_{11})=g(e_1(p_{11}+q_{21}))=g(e_1)(p_{11}+q_{21})+\varphi(e_1)D(p_{11}+q_{21})\\
		\& ~& g(q_{21})=g(e_2(p_{11}+q_{21}))=g(e_2)(p_{11}+q_{21})+\varphi(e_2)D(p_{11}+q_{21}).
	\end{align*}
	Adding these,
	\begin{align*}
		g(p_{11})+g(q_{21})&=(g(e_1)+g(e_2))(p_{11}+q_{21})+(\varphi(e_1)+\varphi(e_2))D(p_{11}+q_{21})\\
		&=g(1)(p_{11}+q_{21})+\varphi(1)D(p_{11}+q_{21})~\text{(By Lemma \ref{lem4.4})}\\
		&=g(1(p_{11}+q_{21}))=g(p_{11}+q_{21}).
	\end{align*}		
	Similarly, we can prove that
	\begin{align*}
		g (p_{22}+q_{12} )=g (p_{22} )+g (q_{12} ) .
	\end{align*}
\end{proof}

\begin{lemma}
	\label{lem4.6}
	Let $p_{ij}, q_{ij}, c_{ij} \in R_{ij}$. Then
	$$
	\begin{aligned}
		(i) ~ g (p_{12}+q_{12} c_{22} )=g (p_{12} )+g (q_{12} c_{22} ),\\
		(ii) ~ g (p_{21}+q_{22} c_{21} )=g (p_{21} )+g (q_{22} c_{21} ).
	\end{aligned}
	$$
\end{lemma}

\begin{proof}
	Note that,
	\begin{align*}
		p_{12}+q_{12} c_{22}= (e_1+q_{12} ) (p_{12}+c_{22} ).	
	\end{align*}		
	Therefore,
	$$
	\begin{aligned}
		g(p_{12}+q_{12} c_{22})
		=&g( (e_1+q_{12} ) (p_{12}+c_{22} ) ) \\
		=& g (e_1+q_{12} ) (p_{12}+c_{22} )+\varphi (e_1+q_{12} ) D (p_{12}+c_{22} ) \\
		=& (g (e_1) + g(q_{12} )) (p_{12}+c_{22} )+ (\varphi (e_1 )+\varphi (q_{12} )) (D (p_{12} )+D (c_{22} ) ) \\
		&\text{(By Lemma \ref{lem4.5})}\\
		=& g (e_1 ) p_{12}+\varphi (e_1 ) D (p_{12} )+g (e_1 ) c_{22}+\varphi (e_1 ) D (c_{22} ) \\
		& +g (q_{12} ) p_{12}+\varphi (q_{12} ) D (p_{12} )+g (q_{12} ) c_{22}+\varphi (q_{12} ) D (c_{22} ) \\
		=&  g(e_1 p_{12}) +g(e_1 c_{22})+ g (q_{12} p_{12} )+g (q_{12} c_{22} ) \\
		=& g (p_{12}) +g (q_{12} c_{22} )~\text{(By Lemma \ref{lem4.2})}.
	\end{aligned}
	$$
	Similarly, we can prove that
	\begin{align*}
		g (p_{21}+q_{22} c_{21} )=g (p_{21} )+g (q_{22} c_{21} ).
	\end{align*}
\end{proof}

\begin{lemma}
	\label{lem4.7}
	Let $p_{ij}, q_{ij} \in R_{ij}$. Then
	$$
	\begin{aligned}
		(i)& ~ g (p_{12}+q_{12} )=g (p_{12} )+g (q_{12} ),\\
		(ii)&~ g (p_{21}+q_{21} )=g (p_{21} )+g (q_{21} ).
	\end{aligned}
	$$
\end{lemma}

\begin{proof}
	Let  $x_{22} \in R_{22}$. Then
	$$
	\begin{aligned}
		& g (p_{12}+q_{12} ) x_{22}+\varphi (p_{12}+q_{12} ) D (x_{22} ) \\
		&=g [ (p_{12}+q_{12} ) x_{22} ]\\
		&= g (p_{12} x_{22}+q_{12} x_{22} ) \\
		&= g (p_{12} x_{22} )+g (q_{12} x_{22} ) ~(\text{By Lemma \ref{lem4.6}})\\
		&=g (p_{12} ) x_{22}+\varphi (p_{12} ) D (x_{22} )+g (q_{12} ) x_{22}+\varphi (q_{12} ) D (x_{22} ).
	\end{aligned}
	$$
	Comparing both sides, we get
	\begin{equation}
		\label{eq4.3}
		\begin{aligned}
			& { [g (p_{12}+q_{12} )-g (p_{12} )-g (q_{12} ) ] x_{22}=0 } \\
			\implies &  { [g (p_{12}+q_{12} )-g (p_{12} )-g (q_{12} ) ]_{12}=0 } \\
			& \& ~ [g (p_{12}+q_{12} )-g (p_{12} )-g (q_{12} ) ]_{22}=0 .
		\end{aligned}
	\end{equation}
	Let $x_{12} \in R_{12}$. Then
	$$
	\begin{aligned}
		& g (p_{12}+q_{12} ) x_{12}+\varphi (p_{12}+q_{12} ) D (x_{12} ) \\
		&=g [ (p_{12}+q_{12} ) x_{12} ]\\
		&= g(0)=0 \\
		&=g (p_{12} x_{12} )+ g(q_{12}x_{12} ) \\
		&= g (p_{12} ) x_{12}+\varphi (p_{12} ) D (x_{12} )+g (q_{12} ) x_{12}+\varphi (q_{12} ) D (x_{12} ),
	\end{aligned}
	$$
	which yields,
	\begin{equation}
		\label{eq4.4}
		\begin{aligned}
			& { [g (p_{12}+q_{12} )-g (p_{12} )-g (q_{12} ) ] x_{12}=0 } \\
			\implies & { [g (p_{12}+q_{12} )-g (p_{12} )-g (q_{12} ) ]_{11}=0 } \\
			& \& ~ [g (p_{12}+q_{12} )-g (p_{12} )-g (q_{12} ) ]_{21}=0 .
		\end{aligned}
	\end{equation}
	Hence, by \eqref{eq4.3} and \eqref{eq4.4},
	\begin{align*}
		g (p_{12}+q_{12} ) = g (p_{12} ) + g (q_{12} ).
	\end{align*}	
	Similarly, we can prove that
	\begin{align*}
		g (p_{21}+q_{21} ) = g (p_{21} ) + g (q_{21} ).
	\end{align*}	
\end{proof}

\begin{lemma}
	\label{lem4.8}
	Let $p_{ii}, q_{ii} \in R_{ii}$. Then
	$$
	\begin{aligned}
		& (i)~ g (p_{11}+q_{11} )=g (p_{11} )+g (q_{11} ),\\
		& (ii)~ g (p_{22}+q_{22} )=g (p_{22} )+g (q_{22} ).
	\end{aligned}
	$$
\end{lemma}

\begin{proof}
	Let $x_{22} \in R_{22}$. Then
	\begin{align*}
		& g (p_{11}+q_{11} ) x_{22}+\varphi (p_{11}+q_{11} ) D (x_{22} )\\
		&=g [ (p_{11}+q_{11} ) x_{22} ]\\
		&=g [ (p_{11}+q_{11} ) x_{22} ]\\
		&=g(0)=0\\
		&=g (p_{11} x_{22} )+g (q_{11} x_{22} )\\
		&=g (p_{11} ) x_{22}+\varphi (p_{11} ) D (x_{22} )+g (q_{11} ) x_{22}+\varphi (q_{11} ) D (x_{22} ).
	\end{align*}
	Comparing both sides,
	\begin{equation}
		\label{eq4.5}
		\begin{aligned}
			& { [g (p_{11}+q_{11} )-g (p_{11} )-g (q_{11}) ] x_{22}=0 } \\
			\implies & { [g (p_{11}+q_{11} )-g (p_{11} )-g (q_{11} ) ]_{12}=0 } \\
			& \& ~ [g (p_{11}+q_{11} )-g (p_{11} )-g (q_{11} ) ]_{22}=0
		\end{aligned}
	\end{equation}
	Similarly, by taking $x_{12} \in R_{12}$,
	\begin{equation}
		\label{eq4.6}
		\begin{aligned}
			& { [g (p_{11}+q_{11} )-g (p_{11} )-g (q_{11} ) ]_{11}=0 } \\
			& \& ~ [g (p_{11}+q_{11} )-g (p_{11} )-g (q_{11} ) ]_{21}=0.
		\end{aligned}
	\end{equation}
	Hence, by \eqref{eq4.5} and \eqref{eq4.6},
	\begin{align*}
		g (p_{11}+q_{11} )=g (p_{11} )+g (q_{11} ).
	\end{align*}		
	Similarly, we can prove that
	\begin{align*}
		g (p_{22}+q_{22} )=g (p_{22} )+g (q_{22} ).
	\end{align*}		
\end{proof}

\begin{lemma}
	\label{lem4.9}
	Let $p_{11} \in R_{11}$, $q_{12} \in R_{12}$, $c_{21} \in R_{21}$ and $d_{22} \in R_{22}$. Then
	$$
	\begin{aligned}
		g (p_{11}+q_{12}+c_{21}+d_{22} )= g (p_{11} )+g (q_{12} ) +g (c_{21} )+g (d_{22} ).
	\end{aligned}
	$$
\end{lemma}

\begin{proof} Let $x_{11} \in R_{11}$. Then
	$$
	\begin{aligned}
		& g (p_{11}+q_{12}+c_{21}+d_{22} ) x_{11} +\varphi (p_{11}+q_{12}+c_{21}+d_{22} ) D (x_{11} ) \\
		&=g [ (p_{11}+q_{12}+c_{21}+d_{22} ) x_{11} ]\\
		&=g (p_{11} x_{11}+c_{21} x_{11} ) \\
		&=g (p_{11} x_{11} )+g (q_{12} x_{11} )+g (c_{21} x_{11} )+g (d_{22} x_{11} )~(\text{By Lemma \ref{lem4.2} and \ref{lem4.5}}) \\
		&=g (p_{11} ) x_{11}+\varphi (p_{11} ) D (x_{11} )+g (q_{12} ) x_{11}+\varphi (q_{12} ) D (x_{11} )\\
		& +g (c_{21} ) x_{11}+\varphi (c_{21} ) D (x_{11} )+g (d_{22} ) x_{11}+\varphi (d_{22} ) D (x_{11} ).
	\end{aligned}
	$$
	Comparing both sides, we get
	\begin{equation}
		\label{eq4.7}
		\begin{aligned}
			& { [g (p_{11}+q_{12}+c_{21}+d_{22} )-g (p_{11} )-g (q_{12} )-g (c_{21} )-g (d_{22}) ] x_{11}=0 .} \\
			\implies & { [g (p_{11}+q_{12}+c_{21}+d_{22} )-g (p_{11} )-g (q_{12} )-g (c_{21} )-g (d_{22} ) ]_{11}=0 } \\
			\&~ & { [g (p_{11}+q_{12}+c_{21}+d_{22} )-g (p_{11} )-g (q_{12} )-g (c_{21} )-g (d_{22} ) ]_{21}=0 }. \\
		\end{aligned}
	\end{equation}
	Similarly, by taking $x_{22}\in R_{22}$,
	\begin{equation}
		\label{eq4.8}
		\begin{aligned}
			& { [g (p_{11}+q_{12}+c_{21}+d_{22} )-g (p_{11} )-g (q_{12} )-g (c_{21} )-g (d_{22} ) ]_{12}=0 } \\
			\& ~& { [g (p_{11}+q_{12}+c_{21}+d_{22} )-g (p_{11} )-g (q_{12} )-g (c_{21} )-g (d_{22} ) ]_{22}=0 }. \\
		\end{aligned}
	\end{equation}
	Hence, by \eqref{eq4.7} and \eqref{eq4.8},
	\begin{align*}
		g (p_{11}+q_{12}+c_{21}+d_{22} )= g (p_{11} )+g (q_{12} ) +g (c_{21} )+g (d_{22} ).
	\end{align*}
\end{proof}

\begin{proof}[Proof of Theorem \ref{thm4.1}]
	Let $a, b \in R$. Then
	\begin{align*}
		& a=p_{11}+p_{12}+p_{21}+p_{22},\\
		& b=q_{11}+q_{12}+q_{21}+q_{22},~\text{for some}~p_{ij},~q_{ij}\in R_{ij}.
	\end{align*}
	Now,
	$$
	\begin{aligned}
		&g(a+b)=g (p_{11}+p_{12}+p_{21}+p_{22}+q_{11}+q_{12}+q_{21}+q_{22} )\\
		&=g [ (p_{11}+q_{11} )+ (p_{12}+q_{12} )+ (p_{21}+q_{21} )+ (p_{22}+q_{22} ) ]\\
		&=g (p_{11}+q_{11} )+g (p_{12}+q_{12} )+g (p_{21}+q_{21} )+g (p_{22}+q_{22} )~\text{(By Lemma \ref{lem4.9})}\\
		&=g (p_{11} )+g(b_ {11})+g (p_{12} )+g (q_{12} )+g (p_{21} )+g (q_{21} ) +g (p_{22} )+g (q_{22} )\\
		& ~\text{(By Lemma \ref{lem4.7} and \ref{lem4.8})}\\
		&=g (p_{11}+p_{12}+p_{21}+p_{22} )+g (q_{11}+q_{12}+q_{21}+q_{22} )~\text{(By Lemma \ref{lem4.9})}\\
		&=g(a)+g(b).
	\end{aligned}
	$$
	Hence, $g$ is additive.
\end{proof}

\begin{corollary}
	\label{cor4.10}
	Let $R$ be a $2$-torsion free semi-prime ring with a non-trivial idempotent $e$ and satisfies the condition \eqref{eqB} given in cor \ref{cor3.8}. Then every multiplicative generalized skew derivation $g$ over $R$ is additive. Moreover, $g$ is a generalized skew derivation on $R$.
\end{corollary}

\begin{proof}
	We have the desired result by Lemma $1.5$ in \cite{jing2012additivity} and Theorem \ref{thm4.1}.
\end{proof}

\section*{Acknowledgement}
The first author is thankful to the University Grants Commission (UGC), Govt. of India, for financial support under UGC Ref. No. 1256 dated 16/12/2019. All authors are thankful to the Indian Institute of Technology Patna for providing research facilities.

\section*{Declarations}
\textbf{Data Availability Statement}: The authors declare that [the/all other] data supporting the findings of this study are available within the article. \\

\textbf{Competing interests}: The authors declare that there is no conflict of interest regarding the publication of this manuscript.\\

\textbf{Use of AI tools Declaration}: The authors declare that they have not used Artificial Intelligence (AI) tools in the creation of this manuscript.\\

\bigskip
\bigskip


{\footnotesize {\bf Sk Aziz, Om Prakash}\; \\ {Department of
		Mathematics}, {Indian Institute of Technology Patna,} {Patna-801106, India.}\\
	{\tt Email: aziz\_2021ma22@iitp.ac.in, om@iitp.ac.in}\\
	
	{\footnotesize {\bf Arindam Ghosh}\; \\ {Department of
			Mathematics}, {Government Polytechnic Kishanganj,} {Kishanganj-855 116, India.}\\
		{\tt Email: arindam.rkmrc@gmail.com}\\

\end{document}